\documentclass[a4paper,english,10pt]{article}
\usepackage[left=.9in, right=.9in, bottom=1.2in, top=1in,dvips]{geometry}
\makeatletter

\AtBeginDocument{\setlength\abovedisplayskip{5pt plus 2pt minus 1pt}
\setlength\belowdisplayskip{5pt plus 2pt minus 1pt}}

\usepackage[utf8]{inputenc}
\usepackage{lmodern}
\usepackage{indentfirst}
\usepackage[affil-it]{authblk}
\usepackage{rotating}
\usepackage{tocloft}
\usepackage{titlesec}
 \titleformat{\subparagraph}[hang]{\normalfont}{\thesubparagraph}{0pt}{\underline}
 \titleformat{\paragraph}[hang]{\normalfont}{\theparagraph}{0pt}{\myuline}
\usepackage{titletoc}
\usepackage{color, soul}
 \usepackage{xcolor}
 \usepackage[citebordercolor={green}]{hyperref}
\usepackage{array}
\usepackage{tabularx}
\usepackage{amssymb}
\usepackage{amsmath}
\usepackage{bbm}

\usepackage{amsthm,amsfonts,amssymb,euscript}
\usepackage{latexsym, multicol, fancybox}
\usepackage{amsmath, amsthm, amssymb, bm}
\usepackage{caption}
\usepackage{psfrag}
\usepackage{setspace}
\usepackage{mathrsfs}
\usepackage{epsfig}
\usepackage{verbatim}
\usepackage[affil-it]{authblk}
\usepackage{ipa}
\usepackage{tikz}

\newtheorem{theorem}{Theorem}[section]
\newtheorem{lemma}[theorem]{Lemma}
\newtheorem{proposition}[theorem]{Proposition}
\newtheorem{corollary}[theorem]{Corollary}

\newtheorem{remark}[theorem]{Remark}

\newtheorem*{theorem*}{Theorem}
\newtheorem*{prop*}{Proposition}

\newcommand{\bea}{\begin{eqnarray}}
\newcommand{\eea}{\end{eqnarray}}
\def\beaa{\begin{eqnarray*}}
\def\eeaa{\end{eqnarray*}}

\newcommand{\MM}{\mathcal{M}}
\def\DD{{\mathcal D}}
\def\TT{{\mathcal T}}
\def\PP{\mathcal{P}}

\def\D{{\bf D}}

\def\g{{\bf g}}

\def\RR{\mathcal{R}}

\def\a{{\alpha}}

\def\b{{\beta}}

\def\De{\Delta}

\def\om{\omega}
\def\vphi{\varphi}

\def\th{\theta}

\def\nab{\nabla}

\newcommand{\nabb}{\nab\mkern-13mu /\,}

\def\Lb{{\,\underline{L}}}

\def\pr{{\partial}}

\def\c{\cdot}

\def\div{{\mbox div\,}}

\def\lab{\label}

\def\OO{\mathcal{O}}

\def\QQ{\mathcal{Q}}
\def\LL{\mathcal{L}}

\def\EE{\mathcal{E}}
\def\AA{\mathcal{A}}
\def\VV{\mathcal{V}}

\def\FF{\mathcal{F}}
\def\UU{\mathcal{U}}
\def\piX{\, ^{(X)}\pi}
\def\That{\widehat{T}}

\def\HH{\mathcal{H}}

\def\e{{\bf e}}

\def\g{{\bf g}}

\def\lz{{\bf \ell_z}}

\allowdisplaybreaks

\begin{document}

 \title{\LARGE \textbf{Physical-space estimates for axisymmetric waves on extremal Kerr spacetime}}
 
 \author[1]{{\Large Elena Giorgi\footnote{elena.giorgi@columbia.edu}}}

 \author[2]{{\Large Jingbo Wan\footnote{jw3976@columbia.edu}}}

\affil[1,2]{\small  Department of Mathematics, Columbia University \vspace{0.2cm} \ }

\maketitle

\begin{abstract} 
We study axisymmetric solutions to the wave equation $\square_g \psi=0$ on extremal Kerr backgrounds and obtain integrated local energy decay (or Morawetz estimates) through an analysis \textit{exclusively in physical-space}. Boundedness of the energy and Morawetz estimates for axisymmetric waves in extremal Kerr were first obtained by Aretakis \cite{Aretakis2012} through the construction of frequency-localized currents used in particular to express the trapping degeneracy. Here we extend to extremal Kerr a method introduced by Stogin \cite{Stogin} in the sub-extremal case, simplifying Aretakis' derivation of Morawetz estimates through purely classical currents.

\end{abstract}

\bigskip\bigskip


\section{Introduction}

The study of the Cauchy problem for the wave equation
\bea\label{eq:wave-intro}
\square_g \psi =0,
\eea
 where $g$ is given by a black hole solution to the Einstein equation, is a topic that has been extensively studied in the past two decades. One of the most important black hole solutions is the vacuum \textit{Kerr family}, a 2-parameter family of solutions $(\mathcal{N}_{M,a}, g_{M,a})$ with $|a|\leq M$, representing a stationary and rotating black hole.
Boundedness and decay properties for solutions to the wave equation on Kerr have been obtained in numerous works in the past two decades, see already Section \ref{sec:previous-work} for an overview of previous results. Numerous such works rely on the derivation of integrated local energy decay estimates, or Morawetz estimates, through an analysis in physical- or frequency-space, or the use of pseudo-differential operators.

We consider here the case of axially symmetric solutions to the wave equation on extremal Kerr backgrounds, corresponding to $|a|=M$. Even though instability properties hold for solutions to the wave equations in extremal black holes as shown by Aretakis \cite{extremal-1}\cite{extremal-2}\cite{Aretakis2015}, boundedness of the energy as well as integrated local energy and pointwise decay estimates have been obtained by Aretakis \cite{Aretakis2012} for axially symmetric waves in extremal Kerr. In \cite{Aretakis2012}, Aretakis adapted the method used by Dafermos-Rodnianski in \cite{DR10} relying on the separability of the wave equation to construct frequency-localized currents. For general solutions to the wave equation, it is in fact not possible to obtain positive definite spacetime estimates through classical energy currents, or in physical-space, as shown by Alinhac \cite{Alinhac}. This is related to the complicated structure of trapped null geodesics for $|a| \neq 0$, whose (Boyer-Lindquist) constant $r$-value covers an open range of values. On the other hand, for axially symmetric solutions the trapping degeneracy collapses to a hypersurface in physical-space and, as observed by Aretakis in the introduction of \cite{Aretakis2012},  ``the obstruction uncovered by Alinhac \cite{Alinhac} does not apply to the axisymmetric case and thus one could in principle expect to derive integrated decay for the full range $|a|\leq M$ using purely classical currents; this remains however an open problem." 

In this paper, we address this problem by deriving integrated local energy estimates for axially symmetric solutions to the wave equation on extremal Kerr exclusively through a \textit{physical-space analysis}. Here by physical-space estimates we refer to an analysis of the wave equation which does not require a mode or frequency decomposition and involves only differential operators.

Recall that in extremal Kerr the event horizon lies at $r=M$ and the effective photon sphere lies at $r_{trap}=(1+\sqrt{2})M$. The degenerate energy for solutions to \eqref{eq:wave-intro} is given by
\beaa
E^{(T)}[\psi](0)&=& \int_{\Sigma_0} |\partial_t\psi|^2+ \left( 1-\frac{M}{r}\right)^2 |\partial_r \psi|^2 + |\nabb\psi|^2,
\eeaa
where $|\nabb \psi|^2=\frac{1}{r^2}|\nabb_{\mathbb{S}^2} \psi|^2$ with $|\nabb_{\mathbb{S}^2} \psi|^2$ the norm of the gradient of $\psi$ on the unit sphere with respect to the standard metric, and in what follows $(t, r, \theta, \phi)$ denote the Boyer-Lindquist coordinates. We prove the following.

\begin{theorem}\label{main-theorem} Let $\psi$ be a sufficiently regular axisymmetric solution to the wave equation in extremal Kerr spacetime with initial data on a spacelike hypersurface $\Sigma_0$ which decays sufficiently fast and let $r_e>M$. Then the following Morawetz estimates:
\bea\label{eq:main-theorem}
\int_{r\geq r_e} \frac{1}{r^3}\left(1-\frac{M}{r}\right)^2|\partial_r \psi|^2+ \frac{1}{r}\left(1-\frac{r_{trap}}{r}\right)^2\left( \frac{1}{r^2}(\partial_t \psi)^2+ |\nabb \psi|^2\right) +\frac{1}{r^4}\left(1-\dfrac{M}{r}\right)^2|\psi|^2  \leq C E^{(T)}[\psi](0),
\eea
where $C$ only depends on $M$ and $r_e$,  can be obtained through {\normalfont{exclusively a physical-space analysis}}. 
\end{theorem}

As a corollary of our main result we recover the following bound that appeared as the crucial Proposition 12.5.1 in \cite{Aretakis2012}, which summarized the results involving frequency decomposition, reflecting the fact that the microlocalization in \cite{Aretakis2012} was only needed in a spatially compact region located away from the horizon.

\begin{corollary}\label{corollary} Let $\psi$ be a sufficiently regular axisymmetric solution to the wave equation in the extremal Kerr spacetime with initial data on $\Sigma_0$ which decays sufficiently fast and let $R_{e}> r_{e} >M$. Then the following Morawetz estimates:
\bea\label{eq:main-theorem}
\int_{r_{e} \leq r \leq R_{e}} \Big( |\partial_{r*} \psi|^2 + |\psi|^2+ \big(r-r_{trap} \big)^2 \big( |\partial_t \psi|^2+|\nabb \psi|^2  \big)  \Big) \leq C E^{(T)}[\psi](0),
\eea
where $r^*=\int \frac{r^2+a^2}{\De}$ and $C$ only depends on $M$, $r_e$ and $R_{e}$,  can be obtained through {\normalfont{exclusively a physical-space analysis}}.
\end{corollary}

The proof of Proposition 12.5.1 in \cite{Aretakis2012} is based on the separation of variables for the solution of \eqref{eq:wave-intro} and an analysis in frequency-space relying on a series of involved microlocal currents tailored to the different regions of validity in frequency space. On the other hand, we obtain Theorem \ref{main-theorem}, and consequently Corollary \ref{corollary}, through the definition of one physical-space current, resulting in a considerable simplification of the construction. In particular, the multiplier used here is uniform in all angular frequencies.

In \cite{Aretakis2012}, Aretakis combined the above result as stated in Corollary \ref{corollary} with positive-definite currents near null infinity and near the horizon and with a uniform boundedness statement of energy, both of which were obtained in \cite{Aretakis2012} in physical-space, to deduce a complete integrated local energy decay.  Finally, by applying Dafermos-Rodnianski's $r^p$-method \cite{DR09a}, Aretakis \cite{Aretakis2012} improved the decay towards null infinity of the integrated local energy decay and used the improved decay to obtain pointwise decay for the solution.  Since these proofs were obtained in \cite{Aretakis2012} through exclusively a physical-space analysis we will not rederive them here, with the exception of the boundedness of the degenerate energy in Section \ref{sec:boundedness-energy}, and refer to \cite{Aretakis2012} for details. In particular, by combining Theorem \ref{main-theorem} with the steps mentioned above obtained by Aretakis in \cite{Aretakis2012}, one can recover the full results of pointwise and power-law energy decay for axially symmetric waves in extremal Kerr exclusively in physical-space.

\subsection{Previous works}\label{sec:previous-work}

We recall here the main results and techniques used in the analysis of the wave equation \eqref{eq:wave-intro} and related stability problems in black hole solutions.

Stability results for the wave equation on Schwarzschild spacetime, corresponding to the case of $a=0$, have been first obtained by Kay-Wald \cite{KW}, who derived a statement of energy boundedness. In the following decades, such statement has been refined to include local energy decay estimates, also known as Morawetz estimates \cite{Mor1}, which give control over a positive-definite spacetime norm through the use of a current associated to the radial vectorfield, as in Blue-Soffer \cite{BS}\cite{B-S2}, Blue-Sterbenz \cite{B-St}, Dafermos-Rodnianski \cite{DR09},  Marzuola-Metcalfe-Tataru-Tohaneanu \cite{MaMeTaTo}. The estimates in this case are obtained as a modified version of the classical  Morawetz  integral energy decay  estimate through the use of a vectorfield of the form $\mathcal{F}(r) \partial_r$, with $\mathcal{F}$ vanishing at $r=3M$, which is the location of orbital null geodesics in Schwarzschild called the \textit{photon sphere}. Also, in \cite{DR09} Dafermos-Rodnianski introduced a vectorfield estimate which captures the so-called \textit{redshift effect}, allowing for pointwise estimates along the event horizon.

In the case of Kerr spacetime with $|a| \neq 0$, the orbital null geodesics are not confined to a hypersurface in physical-space, but cover an open region in (Boyer-Lindquist) $r$-value which depends on the energy and angular momentum of the geodesics. Moreover, the stationary Killing vectorfield $\partial_t$ fails to be timelike in an ergoregion, and therefore its associated conserved energy is not positive definite, in a phenomenon called \textit{superradiance}.   The analysis of the wave equation is complicated by the presence of the ergoregion and the dependence of the trapping region on the frequency of the solution, as the the high frequency obstruction to decay given by the trapping region cannot be described by the classical vectorfield method as shown by Alinhac \cite{Alinhac}. 
For this reason,  the derivation of    a Morawetz estimate   in this case requires  a more refined  analysis involving     both the vectorfield method and   mode  decompositions or pseudo-differential operators.

The mode decomposition refers to the analysis of mode solutions of the  separated form 
\begin{equation}\label{mode-solution}
\psi(r, t, \theta, \phi)= e^{- i \omega t} e^{i m\phi} R(r) S(\theta), \qquad \omega \in \mathbb{R}, \qquad m \in \mathbb{Z}
\end{equation}
which is related to the Fourier transform of the solution with respect to the symmetries of the spacetime, and corresponds to its frequency decomposition. The presence of an additional hidden symmetry of the spacetime, known as the  \textit{Carter tensor} \cite{Carter}, allows to reduce the study of the wave equation to the respective radial and angular ODEs for the functions $R(r)$ and $S(\theta)$. Such frequency-analysis has been developed by Dafermos-Rodnianski \cite{DR10} and Dafermos-Rodnianski-Shlapentokh-Rothman \cite{DRSR} in sub-extremal Kerr, where frequency-dependent multiplier currents for the separated solutions are carefully constructed using the structure of trapping (which stays localized in frequency-space) and the fact that superradiant frequencies are not trapped. This allows for the construction of a frequency-space analogue of the current $\mathcal{F}(r) \partial_r$ which vanishes at a different $r_{trap}$ for each set of trapped frequencies \cite{DR11}\cite{DR11b}\cite{DR13}. Remarkably, the frequency-space analysis in \cite{DRSR} is the only one among the techniques mentioned here which holds in the full sub-extremal range $|a|<M$.
 Observe that to justify the separation of general solutions into \eqref{mode-solution} through a Fourier transform, one needs to require square integrability in time, which can be proved to hold through a continuity argument in $a$. 

The use of pseudo-differential operators appeared in the work of Tataru-Tohaneanu \cite{Tataru}, where they made use of a pseudo-differential modification of the vectorfield $\FF(r) \partial_r$ which was differential in $\partial_t$, and with a kernel supported in a small neighborhood of $r=3M$. The pseudo-differential operator is constructed perturbatively from the choices of vectorfield and functions in Schwarzschild given in \cite{MaMeTaTo}, and therefore yields local energy decay estimates for slowly rotating Kerr spacetime only.

Despite Alinhac's obstruction \cite{Alinhac}, Andersson-Blue \cite{And-Mor} obtained integrated local energy estimates for the equation in slowly rotating Kerr spacetime exclusively in physical space by generalizing the vectorfield method.
Andersson-Blue's method makes use of the Carter hidden symmetry in Kerr as a physical-space commutator to the wave equation. This allows to obtain a local energy decay identity at the level of three derivatives of the solution which degenerate near $r=3M$, as trapped null geodesics lie within $O(|a|)$ of the photon sphere $r=3M$ for slowly rotating Kerr black holes. Such physical-space estimates have the  usual advantages of  the classical  vectorfield method, such as being robust with respect to perturbations of the metric, see \cite{GKS}\cite{Giorgi9}.

The geometry of the extremal Kerr spacetime satisfying $|a|=M$ (or extremal Reissner-Nordstr\"om with $|Q|=M$) exhibits several qualitative differences from the sub-extremal case, most notably the degeneration of the red-shift effect at the horizon due to the fact that the surface gravity vanishes. In extremal Kerr for generic solutions to the wave equation certain higher order derivatives asymptotically blow up along the event horizon as a consequence of conservation laws discovered by Aretakis \cite{extremal-1}\cite{extremal-2}\cite{Aretakis2015}, in what is now known as the \textit{Aretakis instability}. This generic blow up is unrelated to superradiance and holds even for axially symmetric solutions, see also \cite{AAG}\cite{AAG1}\cite{AAG2}.

Axially symmetric solutions to the wave equation, both in the sub-extremal and the extremal case, present two major simplifications: superradiance is effectively absent and the trapping region collapses to a physical-space hypersurface. The conserved energy associated to $\partial_t$ is positive-definite for axially symmetric solutions even though the energy is degenerate\footnote{There is however a way to capture in a quantitative way the degenerate redshift close to the event horizon in the case of extremal Reissner-Nordstr\"om as shown in \cite{AAG0}.} along the event horizon (see Section \ref{sec:boundedness-energy}). 
Also, in axial symmetry the orbital null geodesics all asymptote towards a unique hypersurface $\{ r=r_{trap}\}$ in physical-space, where $r_{trap}$ is defined as the unique root in the exterior region of the polynomial 
\begin{equation}\label{definition-TT-intro}
 \mathcal{T}:= r^3-3Mr^2 +  a^2r+Ma^2.
\end{equation}
In this case, the construction of the current $\mathcal{F}(r) \partial_r$ simplifies (see \cite{DR10}) and can in principle be performed in physical-space. In \cite{Stogin}, Stogin constructed a current in physical space which yields positivity of the local integrated energy estimates in the full sub-extremal range $|a|<M$. Notice that to obtain positivity of the zero-th order term, Stogin uses the non-degenerate redshift effect which is absent in the extremal case. Stogin's construction \cite{Stogin} also applies to wave maps in Kerr spacetime, which was earlier treated in \cite{IK}.

In the case of axially symmetric solutions in extremal Kerr, Aretakis \cite{Aretakis2012} proved integrated local energy decay, energy and pointwise uniform boundedness of solutions and power-law energy and pointwise decay of solutions, all of them up to and including the event horizon. The derivation of the integrated local energy decay in \cite{Aretakis2012} uses an adaptation of the frequency-analysis of Dafermos-Rodnianski \cite{DR10}, which require novel microlocal currents allowing to decouple the Morawetz estimates from the (degenerate) redshift. As in  \cite{DR10}, to justify the Fourier transform in time, cut-off in time are needed which create error terms that have to be controlled by auxiliary microlocal currents in addition to novel classical vectorfields, resulting in intricate constructions to obtain positivity of the spacetime energy.

Finally, we remark here that the advances developed for the study of the wave equation have been used in the analysis of the Einstein equation in various settings, see \cite{Ch-Kl}\cite{Lind-Rodn}\cite{HVas2} for perturbations of Minkowski spacetime, see \cite{DHR}\cite{mu-tao}\cite{Johnson}\cite{KS}\cite{DHRT}\cite{Benomio} for perturbations of Schwarzschild spacetime, see \cite{Giorgi4}\cite{Giorgi5}\cite{Giorgi6}\cite{Giorgi7a} for perturbations of sub-extremal Reissner-Nordstr\"om and the recent \cite{Mario} for perturbations of extremal Reissner-Nordstr\"om, see \cite{Whiting}\cite{Yakov}\cite{ma2}\cite{TeukolskyDHR}\cite{ABBMa2019}\cite{Rita}\cite{ABBMa2021}\cite{Kerr-lin2}\cite{KS-GCM1}\cite{KS-GCM2}\cite{Y-R}\cite{KS:Kerr}\cite{Shen}\cite{GKS} for perturbations of Kerr, see \cite{Civin}\cite{Giorgi8}\cite{Giorgi9} for perturbations of Kerr-Newman.
In the case of positive cosmological constant, see \cite{DR07}\cite{Hintz-Vasy}\cite{Hintz-M}\cite{Georgios1}\cite{Georgios2}.

\subsection{Overview of the result}

We give here an overview of the proof of Theorem \ref{main-theorem}. We apply the vectorfield method to the current associated to a vectorfield $X$, a scalar function $w$ and a one-form $J$:
 \beaa
 \PP_\mu^{(X, w, J)}[\psi]&=&\QQ[\psi]_{\mu\nu} X^\nu +\frac 1 2  w  \psi \pr_\mu \psi   -\frac 1 4(\pr_\mu w )|\psi|^2+\frac 1 4 J_\mu |\psi|^2,
  \eeaa
  where $\QQ[\psi]_{\mu\nu}$ is the energy-momentum tensor associated to a solution to the wave equation
\beaa
\QQ[\psi]_{\mu\nu}&=& \pr_\mu\psi \pr_\nu \psi -\frac 1 2 \g_{\mu\nu} \pr_\lambda \psi \pr^\lambda \psi.
\eeaa
In order to derive Morawetz estimates, we use the vectorfield $X=\FF(r) \partial_r$, scalar function $w$ and one-form $J$ given by
\beaa
 \FF=zu, \qquad w=z\partial_r u, \qquad J=v\partial_r, 
 \eeaa
 where $z(r)$, $u(r)$, and $v(r)$ are well-chosen functions of $r$, so that the divergence of the current can be written as 
   \bea\label{eq:divergence-PP-intro}
   |q|^2 \D^\mu \PP_\mu^{(X, w, J)}[\psi] &=&\AA |\pr_r\psi|^2 + \UU^{\a\b}(\pr_\a \psi )(\pr_\b \psi )+\VV |\psi|^2+\frac 1 4|q|^2 \div(J |\psi|^2),
   \eea
   where $|q|^2=r^2+a^2\cos^2\th$ and $\partial_\alpha$, $\partial_\beta$ indicate $\partial_t$, $\partial_\th$, $\partial_\phi$,
  see Lemma \ref{proposition:Morawetz1} for the expression of the coefficients $\AA, \UU, \VV$. The axial symmetry of the solution crucially allows to simplify further the principal term $\UU^{\a\b}(\pr_\a \psi )(\pr_\b \psi )$, which for $z=\frac{\De}{(r^2+a^2)^2}$ is given by 
  \beaa
\UU^{\a\b}(\pr_\a \psi )(\pr_\b \psi )
&=&   \frac{u \TT}{(r^2+a^2)^3}\,  |q|^2 |\nab \psi|^2,
\eeaa
where $|\nab\psi|^2$ is defined by \eqref{eq:O-nab} and $\TT$ as in \eqref{definition-TT-intro}.
 
 For the choice of functions $z$, $u$, $w$, we adapt a construction introduced by Stogin \cite{Stogin} in sub-extremal Kerr for $|a|<M$ (also subsequently used and adapted in \cite{KS}\cite{Giorgi4}\cite{Giorgi7a}). In \cite{Stogin} the function $u$ is defined in terms of $w$ using the relation $w=z \partial_r u$ and required to vanish at $r_{trap}$, the root of the polynomial $\TT$, while $z$ and $w$ are given respectively by the geodesic potential and a 
 differentiable function defined piecewise. 
In the sub-extremal case, such construction implies the non-negativity of the first three coefficients on the right hand side of \eqref{eq:divergence-PP-intro}, but still presents remaining issues. In particular, the vectorfield $X=zu \partial_r$ blows up logarithmically towards the horizon as $\log(r-r_{+})\partial_r$ and the coefficient of the zero-th order term $\VV$ vanishes in an interval of $r$ outside the event horizon. Such issues are resolved in \cite{Stogin} by relying on the use of the redshift vectorfield in sub-extremal Kerr: the vectorfield $X$ and function $w$ are modified close to the event horizon to obtain a vectorfield which is regular up to the event horizon (see also \cite{MaMeTaTo}) but such modification introduces a negative contribution in the zero-th order term close to the event horizon. The use of the redshift vectorfield as in \cite{DR09} is then used to fix the degeneracy of the $|\partial_r\psi|^2$ at the event horizon, which is then used along with an integrated local Hardy estimate to obtain positivity of the zero-th order term.

For extremal Kerr, we set
   \beaa
z=\frac{(r-M)^2}{(r^2+M^2)^2},
\eeaa
and explicitly define the differentiable function $w$,  see \eqref{eq:def-w}.  In this case, the vectorfield $X$ behaves like $(r-M)^2\log(r-M)\partial_r$, which also does not admit a regular extension towards the event horizon, and the zero-th order term still vanishes in an interval of $r$, for which a non-degenerate redshift estimate cannot be used as is absent in extremality.
We rely instead on a global pointwise Hardy inequality in $r\geq r_e >M$ which degenerates as $r_e \to M$, capturing the degeneracy of the redshift. The Hardy inequality is based on the use of the one-form $J=v \partial_r$ for an explicit function $v$, see \eqref{eq:def-v}, solution to an ODE, that is used to obtain positivity of the zero-th order term $\VV$, fixing the second issue. We finally also add a control of the time derivative (degenerate at trapping) in the integrated local energy estimates by using the Lagrangian of the wave equation to prove Theorem \ref{main-theorem}. The above construction gives a simple alternative proof of Aretakis' result in \cite{Aretakis2012} in physical-space, bypassing the frequency decomposition and addressing the open problem raised by Aretakis \cite{Aretakis2012}.

 Observe that the Morawetz estimates obtained in Theorem \ref{main-theorem} do not hold up to and including the event horizon, due to the aforementioned degeneracy. In order to capture the correct degeneracy of the weight in $(r-M)$ towards the event horizon, one would need improved weights close to the horizon, as for example obtained in \cite{AAG2} in the case of extremal Reissner-Nordstr\"om. For completeness, we obtain the equivalent improved Morawetz estimates for axially symmetric waves in extremal Kerr in Appendix \ref{sec:appendix}.

This paper is organized as follows: in Section \ref{sec:extremal} we recall the main properties of the extremal Kerr spacetime and in Section \ref{section-wave} we review preliminary computations on the wave equation and the vectorfield method. In Section \ref{sec:Morawetz} we extend Stogin's method \cite{Stogin} to derive the Morawetz estimates for axisymmetric waves in extremal Kerr through a physical-space analysis. Finally, in Appendix \ref{sec:appendix} we obtain weighted estimates close to the horizon. 

\bigskip

\noindent\textbf{Acknowledgements.} The authors would like to thank Stefanos Aretakis for his comments on the manuscript. The first author acknowledges the support of NSF Grants No. DMS-2128386 and No. DMS-2306143 and of a grant of the Simons Foundation (825870, EG).

\section{Extremal Kerr spacetime}\label{sec:extremal}

We recall here the main properties of the extremal Kerr spacetime which are relevant to this paper.  In Section \ref{subsection:metric} we introduce the metric in Boyer-Lindquist and Eddington-Finkelstein coordinates and define the differential structure of the manifold. In Section \ref{subsection:killing} we define the Killing vectorfields of the metric and in Section \ref{subsubsection:trapped-geodesics} we recall the properties of trapped null geodesics on extremal Kerr. For a more detailed presentation of properties of extremal black holes see \cite{extremal-1}\cite{extremal-2}\cite{Aretakis2012}\cite{Aretakis2015}.

\subsection{The manifold and the metric}\label{subsection:metric}

The Kerr metric in Boyer-Lindquist (BL) coordinates $(t, r,  \th, \phi)$ takes the form
\bea\label{metric-KN}
\begin{split}
\g_{M, a}&=- \frac{\De-a^2\sin^2\th}{|q|^2}dt^2-\frac{2a\sin^2\th}{|q|^2}\left(  (r^2+a^2)-\De\right)dt d\phi+\frac{|q|^2}{\Delta}dr^2+|q|^2 d\th^2\\
&+\frac{\sin^2\th}{|q|^2}\left((r^2+a^2)^2-\De a^2\sin^2\th \right)d\phi^2 ,
\end{split}
\eea
where 
\beaa
\Delta &=& r^2-2Mr+a^2=(r-r_{+})(r-r_{-}), \qquad |q|^2=r^2+a^2\cos^2\th.
\eeaa
and $r_{\pm}=M \pm \sqrt{M^2-a^2}$.

The Kerr metric represent a stationary and rotating black hole of mass $M$ and angular momentum $Ma$.
For $|a|<M$ the metric describes the sub-extremal Kerr spacetime, for $|a|=M$ the extremal Kerr and for $|a|>M$ the spacetime contains a naked singularity. If $a=0$ we obtain the Schwarzschild solution.

If $|a| \leq M$, to remove the coordinate singularity at $\De=0$ describing the black hole event horizon, one can define the functions
\beaa
r^*=\int \frac{r^2+a^2}{\De}, \qquad \phi^*=\phi+ \int \frac{a}{\De}, \qquad v=t+r^{*}
\eeaa
and obtain the Kerr metric in the ingoing Eddington-Finkelstein coordinates $(v, r, \th, \phi^*)$
\begin{equation}\label{metric-EF}
\begin{split}
\g_{M, a}&=- \frac{\De-a^2\sin^2\th}{|q|^2}dv^2 + 2 dv dr -\frac{2a\sin^2\th\left(  (r^2+a^2)-\De\right)}{|q|^2} dv d\phi^{*} \\
&-2a\sin^2\th dr d\phi^*+|q|^2 d\th^2+\frac{\sin^2\th}{|q|^2}\left((r^2+a^2)^2-\De a^2\sin^2\th \right)(d\phi^{*})^2,
\end{split}
\end{equation}
which is regular at the horizon.

 The principal null vectors which are regular towards the horizon are given in BL coordinates by 
\beaa
e^*_4=\frac{r^2+a^2}{|q|^2} \pr_t +\frac{\De}{|q|^2} \pr_r +\frac{a}{|q|^2} \pr_\phi, \qquad e^*_3   =\frac{r^2+a^2}{\De} \pr_t -\pr_r +\frac{a}{\De} \pr_\phi,
\eeaa
from which we deduce that the radial vectorfield $\frac{\Delta}{|q|^2}\partial_r=\frac 1 2 \big(e_4^* -\frac{\Delta}{|q|^2}e_3^* \big)$ is regular at the horizon.

From the form of the Kerr metric in Boyer-Lindquist coordinates given by \eqref{metric-KN}, one can deduce \cite{And-Mor} that its conformal inverse $|q|^2 \g_{M,a}^{-1}$ can be written as 
\bea\label{inverse-metric-Kerr}
|q|^2 \g_{M,a}^{\a\b}&=& \Delta \partial_r^\a \partial_r^\b+\frac{1}{\Delta} \RR^{\a\b}
\eea
where
\bea
\RR^{\a\b}&=&  -(r^2+a^2)^2 \partial_t^\a \partial_t^\b-2a(r^2+a^2)\partial_t^{(\a} \partial_\phi^{\b)}-a^2  \partial_\phi^\a \partial_\phi^\b+ \Delta O^{\a\b}, \label{definition-RR-tensor}\\
 O^{\a\b}&=& \partial_\th^\a  \partial_\th^\b  +\frac{1}{\sin^2\th} \partial_{\phi}^\a \partial_{\phi}^\b+2a\partial_t^{(\a} \partial_\phi^{\b)}+a^2 \sin^2\th \partial_t^\a \partial_t^\b,
\eea
where $O^{\a\b}$ is related to the hidden Carter symmetry of the Kerr spacetime. 
We denote 
   \bea\label{eq:O-nab}
   O^{\a\b}(\pr_\a \psi )(\pr_\b \psi )=|\partial_\theta \psi|^2+\big|\frac{1}{\sin\theta} \partial_\phi \psi + a\sin\theta \partial_t \psi\big|^2=:|q|^2 |\nab \psi|^2.
   \eea

We now describe the differential structure of the metric. Given standard spherical coordinates $(\theta, \phi^*)$ on the sphere $\mathbb{S}^2$ and $(v, r)$ global coordinates system on $\mathbb{R}^2$, the ambient manifold is defined to be $\mathcal{N}=\{ (v, r, \th, \phi^*)\in \mathbb{R} \times \mathbb{R}\times \mathbb{S}^2 \setminus \{ \mathbb{R} \times \{ 0\} \times S_{eq}\} \}$, where $S_{eq}=\mathbb{S}^2 \cap \{ \th=\frac{\pi}{2}\}$ denotes the equator of the sphere.

In the case of extremal Kerr spacetimes, we have
\bea
\Delta=(r-M)^2
\eea
and the roots of $\Delta=0$ degenerate to $r_{+}=r_{-}=M$. The event horizon is defined by $\mathcal{H}^+=\mathcal{N} \cap \{ r=M\}$, the black hole region corresponds to $\mathcal{N} \cap \{ r <M\}$ and the exterior region (covered by the Boyer-Lindquist coordinates) is given by $\mathcal{D}=\mathcal{N} \cap \{ r> M\}$.

\subsection{The Killing vectorfields}\label{subsection:killing}

The coordinate vectorfields $T=\partial_v$ and $Z=\partial_{\phi^*}$ coincide with the coordinate vectorfields $\partial_t$ and $\partial_\phi$ in Boyer-Lindquist coordinates, which are manifestly Killing for the metric \eqref{metric-KN}. 
The stationary Killing vectorfield $T=\partial_t$ is asymptotically timelike as $r \to \infty$, and spacelike close to the horizon, in the ergoregion $\{ \Delta - a^2\sin^2\th <0\}$.

The vectorfield $\That:=\pr_t+\frac{a}{r^2+a^2} \pr_\phi$ satisfies,
 \bea\label{eq:g-That-That}
   \g_{M, a}(\That, \That)&=& -\frac{\Delta|q|^2}{ (r^2+a^2)^2},
   \eea 
   see for example Proposition 3.2.2 of \cite{GKS}, 
   and is therefore timelike in the exterior region $\mathcal{D}$ and null on the horizon $\mathcal{H}^+$. 
In particular, its restriction to the event horizon, also called the Hawking vectorfield
\beaa
\That_\HH:=\partial_t+\om_\HH \partial_\phi, \qquad \text{with} \quad \om_\HH=\frac{a}{r_{+}^2+a^2},
\eeaa
is a Killing vectorfield which is null and normal to the horizon. 

In the extremal case, the angular velocity $\om_\HH$ of the horizon is given by $\om_\HH=\frac{1}{2M}$ and we have $\nabla_{\That_\HH} \That_\HH=\kappa \That_{\HH}=0$ along the horizon, where $\kappa=\frac{r_{+}-r_{-}}{2(r_{+}^2+a^2)}$ is the surface gravity, which is positive in the sub-extremal range and vanishes in the extremal case.

\subsection{Trapped null geodesics}\label{subsubsection:trapped-geodesics}

In Kerr spacetimes there exist orbital null geodesics, i.e. geodesics for which the radial coordinate $r$ remains constant. Because of the integrability of the geodesic flow due to the presence of the Carter tensor \cite{Carter}, we can give the following characterization of trapped null geodesics in Kerr spacetime. 

\begin{lemma}[Lemma 3.8.3 in \cite{GKS}]\label{lemma:trapped-geodesics} Let $\gamma(\lambda)$ be a null geodesics in a Kerr spacetime whose constant of motions  
\beaa
\e:=-\g(\dot{\gamma}, \partial_t), \qquad \lz:=-\g(\dot{\gamma}, \partial_\phi)
\eeaa
denote its energy and azimuthal angular momentum respectively. 
Then $\gamma$ is an orbital null geodesic if it satisfies
\bea\label{eq:trapped-region-KN}
\TT_{\e, \lz}:= \big( r^3-3Mr^2 + a^2r+Ma^2\big)\e-  (r-M) a\lz=0.
\eea

\end{lemma}

The orbital null geodesics obtained above are trapped, i.e. neither cross the event horizon nor terminate at null infinity.
From \eqref{eq:trapped-region-KN} we can see that for $a=0$ all trapped null geodesics all concentrate at $\{ r=3M\}$, which is the \textit{photon sphere} of Schwarzschild spacetime. On the other hand, for $|a|\neq 0$ there are null geodesics with constant $r$ for an open range of $r$. 
Nevertheless, if $\lz=0$, i.e. for trapped null geodesics orthogonal to the axial Killing vectorfield, the trapped region defined by \eqref{eq:trapped-region-KN} reduces to a hypersurface defined by
\bea\label{definition-TT}
 \TT:= r^3-3Mr^2 +  a^2r+Ma^2 =0.
\eea
Observe that the polynomial $\TT$ has a unique single root in the exterior of the black hole region, and we denote it by $r_{trap}$. Trapped null geodesics constitute an obstruction to decay for the high frequency limit of solutions to the wave equation. For axisymmetric waves, the trapping obstruction simplifies as it concentrates on the hypersurface $\TT=0$ in physical-space, becoming an \textit{effective photon sphere}.

In the extremal Kerr for $|a|=M$, the trapping surface for axisymmetric waves becomes
\beaa
\TT=r^3-3Mr^2 +  M^2r+M^3=(r-M)(r^2-2Mr-M^2)=0
\eeaa
whose root in the exterior is $r_{trap}=(1+\sqrt{2})M$.

\section{Preliminaries}\label{section-wave}

We recall here some preliminaries concerning the wave equation. In Section \ref{subs:wave} we introduce the wave equation operator and the foliation in extremal Kerr and in Section \ref{subs:vfmethod} we recall the main notations of the vectorfield method. In Section \ref{subs:preM} we collect some preliminary computations for the derivation of the Morawetz estimates obtained in Section \ref{sec:Morawetz}.

\subsection{The wave equation}\label{subs:wave}

The wave operator for a scalar function $\psi$ on a Lorentzian manifold is given by
\beaa
\square_\g \psi=\frac{1}{\sqrt{-\det \g}}\partial_\a ((\sqrt{-\det\g}) \g^{\a\b} \partial_\b \psi).
\eeaa
From the expression for the inverse metric \eqref{inverse-metric-Kerr}, we deduce that the wave operator for the Kerr metric in Boyer-Lindquist coordinates $(t, r, \theta, \phi)$ is given by
\begin{equation}\label{square-GKS}
\begin{split}
|q|^2\square_{\g_{M,a}}&=\pr_r(\Delta \pr_r) +\frac{1}{\Delta} \Big(-(r^2+a^2)^2 \pr^2_t-2a(r^2+a^2)\pr_t \pr_\phi-a^2\pr_\phi^2\Big)\\
&+\frac{1}{\sin\th} \pr_\th(\sin\th\pr_\th)+\frac{1}{\sin^2\th} \pr^2_\phi +2a \partial_t\partial_\phi+a^2\sin^2\th\pr^2_t.
\end{split}
\end{equation}
In ingoing Eddington-Finkelstein coordinates $(v, r, \th, \phi^*)$, the wave opeator is given by
\bea\label{wave-EF}
\begin{split}
|q|^2 \square_{\g_{M,a}}&=\pr_r(\Delta \pr_r) +2(r^2+a^2) \pr_v \pr_r+2a \pr_r \pr_{\phi^*}+ 2r \pr_v \\
&+ \frac{1}{\sin\th} \pr_\th(\sin\th\pr_\th)+\frac{1}{\sin^2\th} \pr^2_{\phi^*}+ 2 a \pr_v \pr_{\phi^*}+ a^2\sin^2\th \pr_v^2.
\end{split}
\eea

Let $\Sigma_0$ be a closed connected axisymmetric spacelike hypersurface in $(\DD \cup \HH^+)$ which crosses the event horizon $\HH^+$ and terminates at null infinity.
We define the region $\MM=J^+(\Sigma_0) \cap ( \DD \cup \HH^+)$, and consider the foliation $\Sigma_\tau=\phi_{\tau}^T(\Sigma_0)$, where $\phi_{\tau}^T$ is the flow of $T$. Since $T$ is Killing, the hypersurfaces $\Sigma_\tau$ are all isometric to $\Sigma_0$. We denote by $n_{\Sigma_\tau}$ the future directed unit vector field normal to $\Sigma_\tau$. By convention, along the event horizon $\HH^+$ we choose $n_{\HH^+}=\widehat{T}_{\mathcal{H}}$. We define the regions $\MM(0, \tau)=\cup_{0 \leq \tilde{\tau} \leq \tau} \Sigma_{\tilde{\tau}}$, $\HH^+(0, \tau)=\HH^+\cap \MM(0, \tau)$ and $\mathcal{I}^+(0, \tau)=\mathcal{I}^+ \cap \MM^+(0, \tau)$.

In what follows we consider axisymmetric solutions to the wave equation in extremal Kerr, i.e.
\bea\label{eq:equation-axial-symm-extremal}
\square_{\g}\psi=0, \qquad \partial_\phi \psi=0,
\eea
where $\g$ denotes the metric of the extremal Kerr spacetime. 
We consider the Cauchy problem for the wave equation in $\MM$ with axisymmetric initial data prescribed on $\Sigma_0$,
\beaa
\psi |_{\Sigma_0}=\psi_0 \in H^k_{loc}(\Sigma_0), \qquad n_{\Sigma_0} \psi|_{\Sigma_0}=\psi_1 \in H^{k-1}_{loc}(\Sigma_0),
\eeaa
for $k \geq 2$ and assuming that $\lim_{x \to \mathcal{I}^+} r \psi^2(x)=0$ where $\mathcal{I}^+$ denotes null infinity.
Standard results imply well-posedness for the above Cauchy problem.

\subsection{The vectorfield method}\label{subs:vfmethod}

 The vectorfield method is based on applying the divergence theorem in a causal domain such as $\MM(0, \tau)$, to certain energy currents, which are constructed from the energy momentum tensor.
  The energy-momentum tensor associated to the wave equation $\square_\g \psi=0$ is given by
\bea\label{definition-energy-momentum-tensor}
\QQ[\psi]_{\mu\nu}&=& \pr_\mu\psi \pr_\nu \psi -\frac 1 2 \g_{\mu\nu} \pr_\lambda \psi \pr^\lambda \psi.
\eea
If $\square_\g \psi=0$, the energy-momentum $\QQ[\psi]_{\mu\nu}$ is divergence free.

 Let $X$ be a vectorfield, $w$ be a scalar function and $J$ a one-form. The current associated to $(X, w, J)$ is defined as 
 \bea\label{definition-of-P}
 \PP_\mu^{(X, w, J)}[\psi]&=&\QQ[\psi]_{\mu\nu} X^\nu +\frac 1 2  w  \psi \pr_\mu \psi   -\frac 1 4(\pr_\mu w )|\psi|^2+\frac 1 4 J_\mu |\psi|^2.
  \eea
 The energy associated to  $(X, w, J)$ on the hypersurface $\Sigma_\tau$ is
  \beaa
E^{(X, w, J)}[\psi](\tau)&=& \int_{\Sigma_\tau} \PP^{(X, w, J)}_\mu[\psi] n_{\Sigma_\tau}^\mu ,
\eeaa
where $n_{\Sigma_\tau}$ denotes the future directed timelike unit normal to $\Sigma_\tau$.

A standard computation gives the following divergence of $\PP$ for a solution to the wave equation $\square_g \psi=0$, see for example \cite{KS}\cite{GKS}, 
  \bea
  \label{le:divergPP-gen}
  \D^\mu \PP_\mu^{(X, w, J)}[\psi]= \frac 1 2 \QQ[\psi]  \c\piX-\frac 1 4 \square_\g w |\psi|^2+\frac 12  w (\pr_\lambda \psi \pr^\lambda \psi)+\frac 1 4 \div(J |\psi|^2),
 \eea
 where $\piX_{\mu\nu}=\D_{(\mu} X_{\nu)}$ is the deformation tensor of the vectorfield $X$. Recall that if $X$ is a Killing vectorfield, then $\piX=0$.

By applying the divergence theorem to $\PP_\mu^{(X, w, J)}$ to $\MM(0, \tau)$ one obtains the associated energy identity:
\bea\label{eq:energy-identity}
E[\psi](\tau) + \int_{\mathcal{H}^+(0, \tau)} \PP_\mu[\psi] n_{\mathcal{H}^+}^\mu + \int_{\mathcal{I}^+(0, \tau)} \PP_\mu[\psi] n_{\mathcal{I}^+}^\mu +\int_{\MM(0, \tau)}\D^\mu \PP_\mu[\psi] = E[\psi](0),
\eea
where we suppressed the superscript $(X, w, J)$ in $E[\psi](\tau)=E^{(X, w, J)}[\psi](\tau), \PP_\mu[\psi]=\PP_\mu^{(X, w, J)}[\psi]$ and the induced volume forms are to be understood. By convention, along the event horizon we choose $n_{\mathcal{H}^+}=T+\frac{a}{M^2+a^2}Z$.

\subsection{Preliminary computations for the Morawetz estimates}\label{subs:preM}

In deriving Morawetz estimates for the wave equation we make use of  the vectorfield $X=\FF(r) \pr_r$, for a well chosen function $\FF$. 
We collect here some relevant computations (see also \cite{And-Mor}\cite{Stogin}\cite{GKS}) which will be used in the next section.

\begin{lemma} For $X=\FF(r) \pr_r$, we have
\bea
\piX^{\a\b}=|q|^{-2} \Big( 2\De^{3/2}\pr_r \big(\frac{\FF}{\De^{1/2}} \big)\pr_r^\a\pr_r^\b- \FF\pr_r\big(\frac 1 \De\RR^{\a\b}\big) \Big) +|q|^{-2} X\big(|q|^2\big) \g^{\a\b},
\eea
and therefore
         \begin{equation}\lab{eq:PPpi(X)}
         \begin{split}
  |q|^2   \QQ[\psi]  \c\piX     &= 2\De^{3/2}\pr_r \big(\frac{\FF}{\De^{1/2}} \big)|\pr_r \psi|^2- \FF\pr_r\big(\frac 1 \De\RR^{\a\b}\big) \pr_\a\psi \pr_\b \psi  \\
  &+\Big( X\big(|q|^2\big) - |q|^2(\div X) \Big) \pr_\lambda \psi \pr^\lambda \psi.
  \end{split}
    \end{equation}

\end{lemma}
\begin{proof} Using the expression for the inverse metric \eqref{inverse-metric-Kerr}, we compute
\beaa
 \LL_X(|q|^2 \g^{\a\b})&=&\LL_X\big(\De \pr_r^\a \pr_r^\b \big)+\LL_X\big(\frac 1 \De\RR^{\a\b} \big)= X(\De) \pr_r^\a\pr_r^\b+ \De [X, \pr_r]^\a \pr_r^\b +\De \pr_r^\a [X, \pr_r]^\b  +\LL_X\big(\frac 1 \De\RR^{\a\b} \big).
\eeaa
For $X=\FF \partial_r$, we obtain
\beaa
 \LL_X(|q|^2 \g^{\a\b})&=& \FF (\pr_r\De) \pr_r^\a\pr_r^\b+ \De [\FF \pr_r, \pr_r]^\a \pr_r^\b+ \De \pr_r^\a [\FF \pr_r, \pr_r]^\b  +\FF \LL_{\pr_r} \big(\frac 1 \De\RR^{\a\b} \big)\\
 &=& \FF (\pr_r\De) \pr_r^\a\pr_r^\b-2 \De (\pr_r \FF)\pr_r^\a \pr_r^\b +\FF \partial_r \big(\frac 1 \De\RR^{\a\b} \big)\\
  &=&-2\De^{3/2}\pr_r \big(\frac{\FF}{\De^{1/2}} \big) \pr_r^\a\pr_r^\b +\FF \partial_r \big(\frac 1 \De\RR^{\a\b} \big).
\eeaa
 By writing
\beaa
 \piX^{\a\b}&=&-\LL_X\big( |q|^{-2}  |q|^2 \g^{\a\b}\big)=-|q|^{-2} \LL_X\big(  |q|^2 \g^{\a\b}\big)- |q|^2\LL_X\big(|q|^{-2}\big)\g^{\a\b}
 \eeaa
 we obtain the stated expressions for the deformation tensors.

Finally we write
  \beaa
     \QQ[\psi]  \c\piX&=&\piX^{\a\b} \pr_\a\psi \pr_\b \psi - (\div X) \pr_\lambda \psi \pr^\lambda \psi
    \eeaa
    since $\g_{\mu\nu} \piX^{\mu\nu}=\g_{\mu\nu} \D^{(\mu}X^{\nu)}=2\div X$.

\end{proof}

  \begin{lemma}   
        \lab{proposition:Morawetz1}
       Let $z(r)$, $u(r)$, $v(r)$ be functions of $r$. Then for
\bea
\lab{def-w-red-in-fun-FF-00-wave}\label{definition-w-}
X= \FF \partial_r, \qquad \quad \FF=z u,  \qquad \quad  w = z \pr_r u , \qquad \quad J=v \partial_r \lab{Equation:w0}
\eea
 the divergence of $\PP_\mu^{(X, w, J)}[\psi]$ satisfies 
  \bea
  \lab{identity:prop.Morawetz1}
   |q|^2 \D^\mu \PP_\mu^{(X, w, J)}[\psi] &=&\AA |\pr_r\psi|^2 + \UU^{\a\b}(\pr_\a \psi )(\pr_\b \psi )+\VV |\psi|^2+\frac 1 4|q|^2 \div(J |\psi|^2),
   \eea
   where
\bea
 \AA&=&z^{1/2}\Delta^{3/2} \partial_r\left( \frac{ z^{1/2}  u }{\Delta^{1/2}}  \right), \lab{eq:coeeficientsUUAAVV}  \\
  \UU^{\a\b}&=& -  \frac{ 1}{2}  u \pr_r\left( \frac z \De\RR^{\a\b}\right), \lab{eq:coeeficientsUUAAVV2}\\
\VV&=&-\frac 1 4 \pr_r \big(\De \pr_r w \big)= -\frac 1 4 \pr_r\big(\De \pr_r \big(
 z \pr_ru  \big)  \big) ,\lab{eq:coeeficientsUUAAVV3}\\
 \frac 1 4 |q|^2\div(J |\psi|^2)&=&  \frac 1 4 |q|^2\Big( 2 v\psi\c \nab_r \psi + \big(\pr_r v+ \frac{2r}{|q|^2} v\big) |\psi|^2 \Big)\label{eq:coefficient-J}.
\eea

        \end{lemma}

\begin{proof} 
    Using \eqref{le:divergPP-gen} and \eqref{eq:PPpi(X)} we compute for $J=0$
\beaa
|q|^2\D^\mu \PP_\mu^{(\FF\partial_r, w, J=0)}[\psi]
&=& \De^{3/2}\pr_r \big(\frac{\FF}{\De^{1/2}} \big)|\pr_r \psi|^2- \frac 1 2 \FF\pr_r\big(\frac 1 \De\RR^{\a\b}\big) \pr_\a\psi \pr_\b \psi  -\frac 1 4|q|^2 \square_\g w |\psi|^2\\
  &&+\frac 1 2 \Big( X\big(|q|^2\big) - |q|^2(\div X)+ |q|^2 w \Big) \pr_\lambda \psi \pr^\lambda \psi.
\eeaa
By defining an intermediate function $w_{int}$ as
     \beaa
   \frac 1 2 \Big(    X\big( |q|^2\big)- |q|^2 \div X+ |q|^2  w\Big)=\frac 1 2 |q|^2 \Big( |q|^{-2} X\big( |q|^2\big)-\div X +w\Big)=:-\frac 1 2 |q|^2w_{int},
       \eeaa
       and using \eqref{inverse-metric-Kerr} to write
       \beaa
      |q|^2 \pr_\lambda \psi \pr^\lambda \psi&=& |q|^2\g^{\lambda \mu} \pr_\lambda \psi \pr_\nu \psi=\Delta |\pr_r \psi|^2+\frac{1}{\De} \RR^{\a\b}\pr_\a \psi \pr_\b \psi,
       \eeaa
       we obtain
  \beaa
     |q|^2\D^\mu \PP_\mu^{(\FF\partial_r, w, J=0)}[\psi]&=&\AA |\pr_r\psi|^2 + \UU^{\a\b}(\pr_\a \psi )(\pr_\b \psi )+\VV |\psi|^2,
   \eeaa
   where
   \beaa
   \AA&=& \De^{3/2}\pr_r \big(\frac{\FF}{\De^{1/2}} \big)-\frac 1 2w_{int}\Delta\\
   \UU^{\a\b}&=&  -\frac 1 2  \FF\pr_r \left(\frac 1 \De\RR^{\a\b}\right)-\frac 1 2   w_{int}\frac 1 \De \RR^{\a\b}\\
   \VV&=&-\frac 1 4|q|^2 \square_\g  w .
   \eeaa
      Now the above can be written as 
 \beaa
\UU^{\a\b} 
&=&  - \frac{ 1}{2}\FF z^{-1}  \partial_r\left( \frac z \De\RR^{\a\b}\right)+ \frac{ 1}{2}\left(\FF z^{-1}\partial_r z  -w_{int}\right) \frac{ \RR^{\a\b}}{\Delta}.
 \eeaa
Setting $\FF=zu$ for a function $u$, and choosing  $w_{int}=  \FF z^{-1}\partial_r z=u \pr_r z $,  we deduce the stated expression for $\UU^{\a\b}$ in \eqref{eq:coeeficientsUUAAVV2}. 
With such choices for $\FF$ and $w_{int}$, we compute 
\beaa
w &=&  |q|^2 \D_\a\big( |q|^{-2}   \FF\pr_r ^\a \big)-w _{int}=   |q|^2\pr_r\big(|q|^{-2}  \FF\big)+\FF( \D_\a\pr_r^\a) -u \pr_r z\\
&=&  |q|^2\pr_r\big(|q|^{-2}  zu \big)+zu |q|^{-2} \pr_r \big(|q|^2\big) -u \pr_r z= \pr_r\big(  zu \big) -u \pr_r z= z \pr_r u,
\eeaa
where we used that $\D_\a \pr_r^\a= \frac{1}{\sqrt{|\g|}} \pr_r \big( \sqrt{|\g|} \big)=\frac{1}{|q|^2} \pr_r \big(|q|^2\big)$.
We also compute
\beaa
\AA&=&   \partial_r\left(\frac{\FF}{\Delta^{1/2}}  \right) \Delta^{3/2}-\frac 12\Delta   w_{int} =\partial_r\left(\frac{zu}{\Delta^{1/2}}  \right) \Delta^{3/2}-\frac 12\Delta   ( \partial_r z) u\\
&=& \frac 1 2 \partial_rz  \frac{  u}{\Delta^{1/2}}   \Delta^{3/2}+z^{1/2} \partial_r\left(\frac{ z^{1/2}  u}{\Delta^{1/2}}  \right) \Delta^{3/2}-\frac 12\Delta   ( \partial_r z)u =z^{1/2}\Delta^{3/2} \partial_r\left( \frac{ z^{1/2}  u}{\Delta^{1/2}}  \right),
\eeaa
and
\beaa
 |q|^2 \square_\g  w=\pr_r \big(\De \pr_r w \big)= \pr_r\big(\De \pr_r \big(
 z \pr_ru  \big)  \big),
\eeaa
as stated.
Finally, for $J=v \partial_r$ we compute
\beaa
  \D^\mu (|\psi|^2 J_\mu)&=&2  v\psi\c \nab_r \psi +|\psi|^2 \div J.
  \eeaa
Using that $\div J= \frac{1}{|q|^ 2} \pr_r \big( |q|^2  v)= \pr_r v+ \frac{2r}{|q|^2} v$, we obtain the stated identity.

\end{proof}

\subsection{Boundedness of the energy}\label{sec:boundedness-energy}

We show here how to obtain boundedness of the energy associated to $T$ for axially symmetric solutions to the wave equation in extremal Kerr. The statement and the proof already appeared in Section 5.1 in \cite{Aretakis2012}, and in axial symmetry can be proved independently of the Morawetz estimates.

Even though the Killing vectorfield $T$ fails to be everywhere timelike and as a consequence the energy $E^{(T)}[\psi]$ associated to it fails to be non-negative definite, superradiance is effectively absent for axially symmetric solutions. In fact, let $n$ be a vector orthogonal to $Z$. Then for an axially symmetric $\psi$ we have
\beaa
E^{(Z)}[\psi](\tau)=\int_{\Sigma_\tau} \QQ[\psi]_{\mu\nu}Z^\nu n^\mu_{\Sigma_\tau}=\int_{\Sigma_\tau} Z(\psi) n_{\Sigma_\tau}( \psi) -\frac 1 2 \g(Z, n_{\Sigma_\tau}) \pr_\lambda \psi \pr^\lambda \psi=0.
\eeaa
On the other hand, the Hawking vectorfield $\That$ is causal everywhere in the exterior and using that $\QQ[\psi]_{\mu\nu} V_1^\mu V_2^\nu$ is non-negative if $V_1$, $V_2$ are causal this implies
\beaa
E^{(T)}[\psi](\tau)=E^{(\That)}[\psi](\tau) \geq 0,
\eeaa
and similarly  $\int_{\mathcal{H}^+(0, \tau)} \PP^{(T)}_\mu[\psi] n_{\mathcal{H}^+}^\mu, \int_{\mathcal{I}^+(0, \tau)} \PP^{(T)}_\mu[\psi] n_{\mathcal{I}^+}^\mu \geq 0$.

Working in the $(v, r, \theta, \vphi^*)$ coordinates, if $n_{\Sigma_\tau}=n^v T+ n^r Y+n^\vphi Z$, then for axially symmetric solutions,
\beaa
E^{(T)}[\psi](\tau)&=& \int_{\Sigma_\tau} \QQ[\psi]_{\mu\nu}T^\nu n^\mu_{\Sigma_\tau}=\int_{\Sigma_\tau} T(\psi) n_{\Sigma_\tau}( \psi) -\frac 1 2 \g(T, n_{\Sigma_\tau}) \pr_\lambda \psi \pr^\lambda \psi\\
&=&\int_{\Sigma_\tau} n^v|T(\psi)|^2+n^r T(\psi) Y( \psi) -\frac 1 2\big(  n^v \g(T, T)+ n^r  \g(T,Y)+n^\vphi  \g(T,Z) \big) \pr_\lambda \psi \pr^\lambda \psi,
\eeaa
where from \eqref{metric-EF} we deduce that $ \pr_\lambda \psi \pr^\lambda \psi=\frac{1}{|q|^2} \big(a^2\sin^2\th |T\psi|^2+\De |Y \psi|^2+ 2 (r^2+a^2) T(\psi) Y(\psi) \big)+ |\nabb \psi|^2$. Since the only contribution for $|Y \psi|^2$ comes from the term in $ \pr_\lambda \psi \pr^\lambda \psi$, which vanishes at the horizon, to have positivity of the energy we need the coefficient of $T(\psi) Y(\psi)$ to vanish at the horizon too. We therefore obtain
\beaa
E^{(T)}[\psi](\tau)&\sim& \int_{\Sigma_\tau} |T \psi|^2+ \left(1-\frac{M}{r} \right)^2 |Y \psi|^2+ |\nabb \psi|^2.
\eeaa

From the energy identity \eqref{eq:energy-identity} applied to $X=T$, since $\EE^{(T, 0)}[\psi]=0$ we then obtain
\bea\label{eq:energy-boundedness}
E^{(T)}[\psi](\tau) +\int_{\mathcal{H}^+(0, \tau)} \PP^{(T)}_\mu[\psi] n_{\mathcal{H}^+}^\mu+\int_{\mathcal{I}^+(0, \tau)} \PP^{(T)}_\mu[\psi] n_{\mathcal{I}^+}^\mu\leq C E^{(T)}[\psi](0).
\eea

As a consequence of the vanishing of the surface gravity at the horizon, there is no redshift effect in extremal black holes. In particular, as shown in \cite{Aretakis2012}, there is no time invariant timelike vectorfield $N$ such that $\EE^{(N, 0)}[\psi]$ is non-negative on the horizon. 
However, one can still quantitatively capture the degenerate redshift close to the horizon by using a current first introduced in \cite{extremal-1}, and obtain a non-degenerate energy statement once the Morawetz estimate is obtained.

\section{Morawetz estimates}\label{sec:Morawetz}

We provide here the proof of our main result.  In Section \ref{subs:stogin} we recall the method introduced by Stogin \cite{Stogin} to construct the relevant functions in the estimates and extend it to the extremal case and in Section \ref{subs:hardy} we complete the construction with a new adapted global pointwise Hardy inequality and an added trapped control on the time derivative of the solution.

\subsection{Stogin's construction}\label{subs:stogin}

Recall Lemma \ref{proposition:Morawetz1}, according to which for functions $z, u, v$ chosen as in \eqref{def-w-red-in-fun-FF-00-wave}, the divergence of $\PP_\mu^{(X, w, J)}[\psi]$ is given by
  \beaa
   |q|^2 \D^\mu \PP_\mu^{(X, w, J)}[\psi] &=&\AA |\pr_r\psi|^2 + \UU^{\a\b}(\pr_\a \psi )(\pr_\b \psi )+\VV |\psi|^2+\frac 1 4|q|^2 \div(J |\psi|^2),
   \eeaa
where $\AA$, $\UU$ and $\VV$ are given as in \eqref{eq:coeeficientsUUAAVV}, \eqref{eq:coeeficientsUUAAVV2}, \eqref{eq:coeeficientsUUAAVV3}. 

Following a standard choice in derivation of Morawetz estimates (see \cite{DR11}\cite{DR13}\cite{Stogin}\cite{And-Mor}\cite{GKS}), we choose the function $z$ so that the coefficient of $|\partial_t\psi|^2$ vanishes and the coefficient of $|\nab \psi|^2$ degenerates at trapping. From \eqref{eq:coeeficientsUUAAVV2} and \eqref{definition-RR-tensor}, we have for axially symmetric solutions
\beaa
\UU^{\a\b}(\pr_\a \psi )(\pr_\b \psi )&=& -  \frac{ 1}{2}  u \pr_r\left( \frac z \De\RR^{\a\b}\right)(\pr_\a \psi )(\pr_\b \psi )\\
&=&   \frac{ 1}{2}  u \pr_r\left( \frac z \De  (r^2+M^2)^2 \right) |\partial_t \psi|^2- \frac{ 1}{2}  u (\pr_r z) \,  O^{\a\b}(\pr_\a \psi )(\pr_\b \psi ).
\eeaa
So we set
\beaa
z=\frac{(r-M)^2}{(r^2+M^2)^2},
\eeaa
and obtain
\beaa
\UU^{\a\b}(\pr_\a \psi )(\pr_\b \psi )
&=&   \frac{u \TT}{(r^2+M^2)^3}\,  O^{\a\b}(\pr_\a \psi )(\pr_\b \psi )\\
&=&   \frac{u (r-M)(r^2-2Mr-M^2)}{(r^2+M^2)^3}\, |q|^2 |\nab \psi|^2.
\eeaa
Observe that from \eqref{eq:O-nab} we have
\beaa
O^{\a\b}(\pr_\a \psi )(\pr_\b \psi )=|q|^2 |\nab \psi|^2=(\partial_\theta \psi )^2+M^2\sin^2\theta (\partial_t \psi )^2.
\eeaa
Using \eqref{eq:coeeficientsUUAAVV} and \eqref{eq:coeeficientsUUAAVV3} we deduce for such choice of $z$,
\bea
\mathcal{A}=\frac{(r-M)^4}{(r^2+M^2)} \partial_r\big(\frac{u}{r^2+M^2}\big), \qquad \mathcal{V}=-\frac 1 4 \pr_r\Big((r-M)^2 \pr_r w  \Big)
\eea
with
\bea\label{eq:w-u-extremal}
w=
 \frac{(r-M)^2}{(r^2+M^2)^2} \pr_ru.
\eea

The main goal here is to choose the functions $u$, $w$ and $v$ so that the divergence of $\PP_\mu^{(X, w, J)}[\psi]$ is positive definite. For the choice of functions $u$ and $w$ we make use of a construction due to Stogin in the sub-extremal Kerr spacetime, see Lemma 5.2.6 in \cite{Stogin}, also used in \cite{KS} \cite{Giorgi7a} \cite{Giorgi9}. In what follows, we adapt Stogin's construction to the case of extremal Kerr. Stogin's construction fails to obtain a positive definite term for $|\psi|^2$ in the entire exterior region and makes use of the redshift estimate and local integrated Hardy inequality to fix this deficiency. In the extremal case, because of the degenerate redshift estimate, we need a new adapted Hardy inequality that we derive in Section \ref{subs:hardy}.

In Stogin's construction \cite{Stogin}, the relation between $u$ and $w$ in \eqref{eq:w-u-extremal} is used to define $u$ in terms of $w$ and then treat $w$ as the free variable. In order to have the coefficient of $|\nab \psi|^2$ to be non-negative, the function $u$ has to change sign at $r=r_{trap}$ and therefore we impose following Stogin \cite{Stogin}:
\bea\label{eq:def-u}
u(r)=\int_{r_{trap}}^r\frac{(s^2+M^2)^2}{(s-M)^2}w(s) ds.
\eea
Further imposing the positivity of the function $w$ we obtain that $u$ is increasing everywhere and changing sign at $r_{trap}$, which implies that $\UU^{\a\b}(\pr_\a \psi )(\pr_\b \psi )$ is non-negative.

Following Stogin, we now choose the function $w$ in order to have positivity of $\mathcal{A}$, i.e. positivity of $\partial_r\big(\frac{u}{r^2+M^2}\big)$. By defining $\widetilde{\mathcal{A}}:= \frac{(r^2+M^2)^2}{2r}\partial_r\left(\frac{u}{r^2+M^2}\right)$, a straightforward computation shows that
\bea\label{eq:non-trivial-uw}
\partial_r\widetilde{\mathcal{A}}=(r^2+M^2)\partial_r \left(\frac{w (r^2+M^2)^2}{2r(r-M)^2} \right).
\eea
Defining $r_{*}:=(2+\sqrt{3})M$ to be the point attaining the maximum of the function $\frac{2r(r-M)^2}{ (r^2+M^2)^2}$, we define $w$ as the positive $C^1$ function
 \begin{align}\label{eq:def-w}
 	w=\begin{cases}
 	 		\frac{1}{4M}\ \qquad &r\leq r_*\\
 		\frac{2r(r-M)^2}{(r^2+M^2)^2} \ \qquad  &r> r_*.
 	\end{cases}
 \end{align} 
Since $r_{*}$ also attains the minimum of the function $\frac{ (r^2+M^2)^2}{2r(r-M)^2}$, the above construction implies that the function $\frac{w (r^2+M^2)^2}{2r(r-M)^2}$ is constant for $r \geq r_{*}$ and decreasing for $r\leq r_{*}$. From \eqref{eq:non-trivial-uw}, one can deduce the same behavior for $\widetilde{\mathcal{A}}$. We now show that the constant value of this function is positive. We have
\beaa
\widetilde{\mathcal{A}}(r_{*})&=&\frac{(r_{*}^2+M^2)^2}{2r_{*}}\partial_r\left(\frac{u}{r^2+M^2}\right)\Big|_{r=r_{*}}= \frac{(r_{*}^2+M^2)}{2r_{*}}\partial_ru\Big|_{r=r_{*}}-u(r_{*})\\
&=& \frac{(r_{*}^2+M^2)^3}{2r_{*}(r_{*}-M)^2} w(r_{*})-w(r_{*})\int_{r_{trap}}^{r_*}\frac{(r^2+M^2)^2}{(r-M)^2} dr
\eeaa
Observe that since the function $\frac{(r^2+M^2)^2}{(r-M)^2}$ is increasing between $r_{trap}$ and $r_{*}$, we can bound the above by
\beaa
\widetilde{\mathcal{A}}(r_{*})
&>& \frac{(r_{*}^2+M^2)^3}{2r_{*}(r_{*}-M)^2} w(r_{*})-\frac{(r_{*}^2+M^2)^2}{(r_{*}-M)^2} w(r_{*})(r_{*}-r_{trap})\\
&=&  \frac{(r_{*}^2+M^2)^2}{2r_{*}(r_{*}-M)^2} w(r_{*})\Big((r_{*}^2+M^2)-2r_{*}(r_{*}-r_{trap})\Big)\\
&=&  \frac{(r_{*}^2+M^2)^2}{4(r_{*}-M)^2} \big(1+\sqrt{2}-\sqrt{3}\big)=  2 r_{*} \big(1+\sqrt{2}-\sqrt{3}\big)M=c_0 M^2,
\eeaa
where $c_0>0$ is a positive constant, explicitly given by $c_0=2 (2+\sqrt{3}) \big(1+\sqrt{2}-\sqrt{3}\big)$, and where we used that $r_{*}^2+M^2=4r_{*}M$, $(r_{*}-M)^2=2Mr_{*}$ and $r_{*}-r_{trap}=(1+\sqrt{3}-\sqrt{2})M$.

Since $\mathcal{A}=\frac{2r(r-M)^4}{(r^2+M^2)^3}\widetilde{\mathcal{A}}$, the above implies that $\AA$ is non-negative, and more precisely:
\bea\label{eq:bounds-AA}
\AA(r) \geq \frac{2r(r-M)^4}{(r^2+M^2)^3}\widetilde{\mathcal{A}}(r_{*}) \geq \frac{2c_0 M^2r(r-M)^4}{(r^2+M^2)^3}.
\eea

 Finally, we are left to analyze the positivity of $\mathcal{V}$. With the choice of $w$ in \eqref{eq:def-w} we compute explicitly
\begin{align*}
	\partial_r\Big((r-M)^2 \partial_r w \Big)
	=&\begin{cases}
 	0\ &r\leq r_*\\
		-\frac{12M (r-M)^2 (r^4-6M^2r^2+M^4)}{(r^2 + M^2)^4} \  &r>r_*.
	\end{cases}
\end{align*}
Observe that the polynomial $r^4-6M^2r^2+M^4$ is positive for $r>(1+\sqrt{2})M$, and since $r_{*}>(1+\sqrt{2})M$ the above is non-negative everywhere in the exterior region.

Integrating the relation $u'=\frac{(r^2+M^2)^2}{(r-M)^2}w$ from \eqref{eq:def-u} and using \eqref{eq:def-w} we deduce a closed form for $u$:
\begin{align*}
		u=\begin{cases}
  		-\frac{M^3}{r-M}+\frac{5Mr}{4}+\frac{r^2}{4}+\frac{r^3}{12M}+2M^2\log (r-M)+C_1 , \ &r\leq r_* \\
 		r^2+C_2\ &r\geq r_*,
 	\end{cases}
\end{align*}
where $C_1$, $C_2$ are suitable constants, such that $u(r_{trap})=0$ and $u$ is continuous at $r_{*}$.

 Observe that the function $u$ blows up as $r$ approaches the event horizon and the vectorfield $X$ does not admit a regular extension towards the event horizon. On the other hand for $r\geq r_e >M$, we deduce that $\mathcal{A}=\frac{(r-M)^4}{(r^2+M^2)}  \partial_r\Big(\frac{u}{r^2+M^2}\Big)\geq \frac{1}{8M}(r-M)^2$ for $r$ sufficiently close to horizon, but away from it.
We deduce that for $J=0$  and  for $r\geq r_e>M$
  \bea\label{eq:estimate1}
  \begin{split}
   |q|^2 \D^\mu \PP_\mu^{(X, w, J=0)}[\psi] &\gtrsim\dfrac{1}{r}\left(1-\dfrac{M}{r}\right)^2 |\pr_r\psi|^2 + \dfrac{1}{r}\left(1-\dfrac{r_{trap}}{r}\right)^2 |q|^2 |\nab \psi|^2+\frac{1}{r^2}1_{r>r_*} |\psi|^2.
   \end{split}
   \eea

The main issue with the estimate \eqref{eq:estimate1} is the vanishing of the coefficient for the zero-th order for $r\leq r_{*}$. We will now fix this issue with a Hardy inequality adapted to the extremal case.

\subsection{The global Hardy inequality and trapped control of the time derivative}\label{subs:hardy}

Here we make use of the one-form $J$ to obtain positivity of the zero-th order term in the entire exterior region. From \eqref{eq:bounds-AA} and since for $r>M$ the function $\frac{r(r-M)^4}{(r^2+M^2)^3} $ achieves its maximum at $(3+2\sqrt{2})M>r_*=(2+\sqrt{3})M$, we define for any $r_{e} \in (M, r_{*})$ the following minimum $c_1:=\min\limits_{r\in[r_e,r_* ]}\mathcal{A}(r)$. Observe that because of the bound \eqref{eq:bounds-AA}, we have that $c_1>0$, with $c_1 \downarrow 0$ as $r_e \to M$. Then for $r \in [r_e, r_{*}]$ we can use the bound $\AA \geq c_1$ and \eqref{eq:coefficient-J} to obtain
  \beaa
\AA |\pr_r\psi|^2 +\frac 1 4|q|^2 \div(J |\psi|^2)&=&\mathcal{A}|\partial_r \psi |^2+\frac{1}{4}|q|^2\left(2v\psi\partial_r\psi  +\left(\partial_r v+\frac{2r}{|q|^2}v \right)|\psi|^2\right)\\
&=& \mathcal{A}\left(\partial_r \psi +\frac{|q|^2}{4\mathcal{A}}v\psi \right)^2-\frac{|q|^4v^2}{16\mathcal{A}} |\psi|^2+\frac{1}{4}|q|^2\left(\partial_r v+\frac{2r}{|q|^2}v \right)|\psi|^2\\
&\geq&\frac{1}{4}|q|^2\left(\partial_r v+\frac{2r}{|q|^2}v - \frac{|q|^2v^2}{4\mathcal{A}}\right)|\psi|^2\\
&\geq&\frac{1}{4}|q|^2\left(\partial_r v+\frac{2r}{|q|^2}v - \frac{|q|^2v^2}{4c_1}\right)|\psi|^2\\
&=&\frac{c_1}{4}|q|^2\left(\left(\frac{v}{c_1}\right)'+\frac{2r}{|q|^2}\left(\frac{v}{c_1}\right)-\frac{|q|^2}{4}\left(\frac{v}{c_1}\right)^2 \right)|\psi|^2. 
   \eeaa
   We want to find a function $y(r)$ that for $r \in [r_e, r_{*}]$
   \beaa
   	y'+\frac{2r}{|q|^2}y-\frac{|q|^2}{4}y^2>0.
   \eeaa
   Using that  $r^2\leq |q|^2\leq r^2+M^2$ we observe that for a negative function $y(r)< 0$ we can bound
   \beaa
   y'+\frac{2r}{|q|^2}y-\frac{|q|^2}{4}y^2\geq y'+\frac{2}{r}y-\frac{r^2+M^2}{4}y^2.
   \eeaa
   In particular we will look for a negative function in $r \in [r_e, r_{*}]$ satisfying $y'+\frac{2}{r}y-\frac{r^2+M^2}{4}y^2>0$.
   A straightforward computation shows that $y_0(r)=-\frac{4}{r(r+M)(r-M)}$ is a negative solution in $r \in [r_e, r_{*}]$ to the ODE $y_0'+\frac{2}{r}y_0-\frac{r^2+M^2}{4}y_0^2=0$. Let $y(r)=\lambda y_0(r)$ be a rescaling of $y_0$ for any constant $0< \lambda <1$, then
   \beaa
   y'+\frac{2}{r}y-\frac{r^2+M^2}{4}y^2&=& \lambda \big(  y_0 +\frac{2}{r} y_0-\frac{r^2+M^2}{4}y_0^2 \big) +\lambda(1-\lambda)\frac{r^2+M^2}{4}y_0^2\\
   &=&\lambda(1-\lambda)\frac{r^2+M^2}{4}y_0^2>0.
   \eeaa
In particular, for $y(r)=\frac 1 2 y_0(r)=-\frac{2}{r(r+M)(r-M)}<0$, we have  in $r \in [r_e, r_{*}]$
\beaa
  	y'+\frac{2r}{|q|^2}y-\frac{|q|^2}{4}y^2&\geq& \frac{1}{2} (1-\frac{1}{2})\frac{r^2+M^2}{4}y_0^2= \frac{r^2+M^2}{r^2(r+M)^2(r-M)^2},
\eeaa
and therefore for 
\bea\label{eq:def-v}
v(r)=c_1 y(r)=-\frac{2c_1}{r(r+M)(r-M)},
\eea
 we have
     \beaa
\AA |\pr_r\psi|^2 +\frac 1 4|q|^2 \div(J |\psi|^2)&\geq &\frac{c_1}{4}\frac{r^2+M^2}{(r+M)^2(r-M)^2}|\psi|^2. 
   \eeaa
   To conclude, combining the above Hardy inequality with the bound \eqref{eq:estimate1} we can improve it to
     \bea\label{eq:estimate-after-Hardy}
   |q|^2 \D^\mu \PP_\mu^{(X, w, J)}[\psi] &\gtrsim & \frac{1}{r}\left(1-\frac{M}{r}\right)^2 |\pr_r\psi|^2  + \frac{1}{r}\left(1-\dfrac{r_{trap}}{r}\right)^2 |q|^2 |\nab \psi|^2+\frac{1}{r^2}1_{r>r_e} |\psi|^2,
   \eea
   for $r \geq r_e>M$.

   \begin{remark} Observe that the above choices of functions can be easily adapted to the case of extremal Reissner-Nordstr\"om as well. It is easy to check that the choices of $z=\frac{(r-M)^2}{r^4}$, $u=\int_{r_{trap}}^r\frac{s^4}{(s-M)^2}w(s) ds$ with $w$ given by $w=\frac{8}{27M}$ for $r\leq r_*=3M$ and $w=\frac{2(r-M)^2}{r^3}$ for $r>3M$ yield \eqref{eq:estimate1}. Similarly, a global Hardy inequality with the choice of function $v=-\frac{2c_1}{r^2(r-M)}$ implies \eqref{eq:estimate-after-Hardy}.
In particular, this gives an alternative simplified proof of the Morawetz estimates in extremal Reissner-Nordstr\"om, where an angular frequency decomposition, as used in \cite{extremal-1}, is unnecessary. See also \cite{HMV}.
   \end{remark}

   The only term that is missing from the above right hand side to give the integral appearing in \eqref{eq:main-theorem} of Theorem \ref{main-theorem} is the trapped control on the time derivative. 
   For a function of $r$ $w_T$, we have from \eqref{le:divergPP-gen}
   \beaa
 |q|^2 \D^\mu \PP_\mu^{(X=0, w_T, J=0)}[\psi]&=& -\frac 1 4 |q|^2\square_\g w_T |\psi|^2+\frac 12  w_T |q|^2(\pr_\lambda \psi \pr^\lambda \psi)\\
 &=& -\frac 1 4 |q|^2\square_\g w_T |\psi|^2+\frac 12  w_T\big(\Delta |\pr_r \psi|^2+\frac{1}{\De} \RR^{\a\b}\pr_\a \psi \pr_\b \psi \big)\\
  &=&-\frac 12  w_T\frac{(r^2+M^2)^2}{(r-M)^2}|\partial_t \psi|^2+\frac 12  w_T(r-M)^2 |\pr_r \psi|^2+ \frac 12  w_T O^{\a\b}\pr_\a \psi \pr_\b \psi \\
  && -\frac 1 4 |q|^2\square_\g w_T |\psi|^2.
  \eeaa
  We choose $w_T$ to be given by
   \beaa
   	w_T=-\frac{(r-M)^2(r-r_{trap})^2}{r^7},
   \eeaa
   and we have
      \bea\label{eq:estimate-wT}
      \begin{split}
 |q|^2 \D^\mu \PP_\mu^{(X=0, w_T, J=0)}[\psi]  &=\frac 12  \frac{(r-r_{trap})^2(r^2+M^2)^2}{r^7}|\partial_t \psi|^2-\frac 12  \frac{(r-M)^4(r-r_{trap})^2}{r^7} |\pr_r \psi|^2\\
 &- \frac 12 \frac{(r-M)^2(r-r_{trap})^2}{r^7} O^{\a\b}\pr_\a \psi \pr_\b \psi  -\frac 1 4 |q|^2\square_\g w_T |\psi|^2.
 \end{split}
  \eea
 We explicitly compute 
   \beaa
   -\frac 1 4 |q|^2\square_\g w_T&=&   -\frac 1 4 \pr_r ( (r-M)^2 \partial_r w_T)\\
   &=& \frac{(r-M)^2}{2r^9}\Big[3 r^4-3 \left(9+4 \sqrt{2}\right)M r^3+\left(93+68 \sqrt{2}\right)M^2 r^2\\
     &&-7 \left(21+16 \sqrt{2}\right) M^3 r+(56 \sqrt{2}+84)M^4\Big]\\
     &=& \frac{3(r-M)^2}{2r^9}\big(r-x_1M \big)\big(r-x_2M \big)\big(r-x_3M \big)\big(r-x_4M \big),
     \eeaa
     where $1<x_1<x_2<x_3<x_4$ are four roots of 
\beaa
	3 x^4-3 \left(9+4 \sqrt{2}\right) x^3+\left(93+68 \sqrt{2}\right) x^2-7 \left(21+16 \sqrt{2}\right) x+56 \sqrt{2}+84=0.
\eeaa
Even though $-\frac 1 4 |q|^2\square_\g w_T$ can be negative for $r \in [x_1M, x_4M]$, it must have a finite negative lower bound there. In particular, by choosing $r_e \in (M, \min(x_1M, r_{*}))$, there exists a sufficiently small $\delta_T>0$ such that 
\beaa
\frac{1}{r^2} 1_{r>r_e}  -\frac 1 4\delta_T  |q|^2\square_\g w_T \geq \frac{1}{r^2}\left(1-\dfrac{M}{r}\right)^2.
\eeaa

Finally combining \eqref{eq:estimate-after-Hardy} and \eqref{eq:estimate-wT} we deduce for $r\geq r_e>M$
\beaa
 \D^\mu \PP_\mu^{(X, w+\delta_Tw_T, J)}[\psi] &\gtrsim & \frac{1}{r^3}\left(1-\frac{M}{r}\right)^2|\partial_r \psi|^2+ \frac{1}{r}\left(1-\frac{r_{trap}}{r}\right)^2\Big( \frac{1}{r^2}(\partial_t \psi)^2+ |\nab \psi|^2\Big)  \\
	&&+\frac{1}{r^4}\left(1-\dfrac{M}{r}\right)^2|\psi|^2.
\eeaa

We are finally set to apply the divergence theorem to the current $\PP_\mu^{(X, w+\delta_Tw_T, J)}[\psi]$. Observe that, by making use of the simple Hardy estimate
\beaa
\int_{0}^\infty |\psi|^2 dx \lesssim \int_0^\infty x^2 |\partial_x \psi|^2 dx
\eeaa
with $x=r-r_{+}$, see also \cite{And-Mor}, to obtain bounds for the zero-th order term, one can estimate the boundary terms $\int_{\Sigma_\tau}\PP_\mu^{(X, w+\delta_Tw_T, J)}[\psi] n^\mu_{\Sigma_\tau}$ by a large constant times a positive definite energy current, such as
\beaa
E^{(T)}[\psi](\tau)&=& \int_{\Sigma_\tau} |T\psi|^2+ \big( 1-\frac{M}{r}\big)^2 |\partial_r \psi|^2 + |\nabb\psi|^2.
\eeaa
We therefore deduce in the region $r \geq r_e>M$
   \beaa
  &&\int_{0}^{\tau} \int \frac{1}{r^3}\left(1-\frac{M}{r}\right)^2|\partial_r \psi|^2+ \frac{1}{r}\left(1-\frac{r_{trap}}{r}\right)^2\big(\frac{1}{r^2} (\partial_t \psi)^2+ |\nabb \psi|^2\big) +\frac{1}{r^4}\left(1-\dfrac{M}{r}\right)^2|\psi|^2  \\
   &\leq& C \Big( E^{(T)}[\psi](\tau)+\int_{\mathcal{H}^+(0, \tau)} \PP^{(T)}_\mu[\psi] n_{\mathcal{H}^+}^\mu+E^{(T)}[\psi](0) \Big).
   \eeaa
Using the boundedness of the energy statement given in \eqref{eq:energy-boundedness} we conclude the proof of Theorem \ref{main-theorem}.

\appendix

\section{Weighted estimates on the horizon}\label{sec:appendix}

 We first collect some general formulas in Kerr spacetime. From the wave equation in \eqref{square-GKS} we deduce
\bea\label{square-GKS-r*}
\frac{\Delta}{(r^2+a^2)^2}|q|^2\square_{\g_{M,a}}&=&\partial^2_{r^\star} +\frac{2r\Delta}{(r^2+a^2)^2}\partial_{r^\star} - \Big( \pr_t+\frac{a}{r^2+a^2} \pr_\phi\Big)^2+\frac{\Delta}{(r^2+a^2)^2}\OO,
\eea
where $\OO=\frac{1}{\sin\th} \pr_\th(\sin\th\pr_\th)+\frac{1}{\sin^2\th} \pr^2_\vphi +2a \partial_t\partial_\vphi+a^2\sin^2\th\pr^2_t$ and we denote  $\partial_{r^\star}=\frac{\Delta}{r^2+a^2}\partial_r$. Observe that the modified Laplacian $\OO$ is related to the Carter operator in Kerr \cite{Giorgi9}.

We define the radiation field associated to $\psi$ as $\Phi := (\sqrt{r^2+a^2}) \psi$, for which we have
\beaa
|q|^2\square_{\g_{M,a}} \Phi &=&\pr_r(\Delta \pr_r (\sqrt{r^2+a^2}) \psi)+\pr_r(\sqrt{r^2+a^2} )\Delta \pr_r (\psi) +\sqrt{r^2+a^2}|q|^2 \square_\g \psi \\
&=&\big( \frac{r \pr_r\Delta}{ r^2+a^2} + \frac{(-2r^2+a^2) \Delta}{(r^2+a^2)^{2}}\big) \Phi+\frac{2r\Delta }{r^2+a^2}  \pr_r \Phi +\sqrt{r^2+a^2}|q|^2 \square_\g \psi.
\eeaa
Using \eqref{square-GKS-r*} for the left hand side of the above, we obtain
\beaa
 \partial^2_{r^\star} \Phi&=&  \Big( \pr_t+\frac{a}{r^2+a^2} \pr_\phi\Big)^2\Phi-\frac{\Delta}{(r^2+a^2)^2}\OO(\Phi)+V \Phi+\frac{\Delta}{(r^2+a^2)^{3/2}}|q|^2 \square_\g \psi,
\eeaa
where $V$ is given by
\beaa
V&:=& \frac{\Delta}{(r^2+a^2)^2}\big( \frac{r \pr_r\Delta}{ r^2+a^2} + \frac{(-2r^2+a^2) \Delta}{(r^2+a^2)^{2}}\big).
\eeaa

Let $\psi$ be a solution to $\square_g\psi=0$ and $\Phi$ its associated radiation field. Then
\bea\label{eq:wave-radiation-field-2}
\Big( \pr_t+\frac{a}{r^2+a^2} \pr_\phi+\partial_{r^\star}\Big)\Big( \pr_t+\frac{a}{r^2+a^2} \pr_\phi-\partial_{r^\star}\Big)\Phi&=& \frac{\Delta}{(r^2+a^2)^2}\OO(\Phi)+\partial_{r^\star}\big( \frac{a}{r^2+a^2}\big) \pr_\vphi \Phi-V \Phi.
\eea
By defining $\Lb=\frac 1 2\big(\pr_t+\frac{a}{r^2+a^2} \pr_\phi-\partial_{r^\star}\big)$ and $L=\frac 1 2 \big( \pr_t+\frac{a}{r^2+a^2} \pr_\phi+\partial_{r^\star} \big)$, we have
\bea\label{eq:wave-radiation-field-3}
4 L  \Lb \Phi&=& \frac{\Delta}{(r^2+a^2)^2}\OO(\Phi)+\partial_{r^\star}\big( \frac{a}{r^2+a^2}\big) \pr_\phi \Phi-V \Phi.
\eea
For axially symmetric solutions to the wave equation in extremal Kerr, the above reduces to 
\bea\label{eq:radiation-field-final}
4 L \Lb \Phi&=& \frac{(r-M)^2}{(r^2+M^2)^2}\OO(\Phi )+O((r-M)^3)\Phi.
\eea
for $\OO=\frac{1}{\sin\th} \pr_\th(\sin\th\pr_\th)+M^2\sin^2\th\pr^2_t$.

Here we obtain improved weighted estimates on the event horizon by adapting the ones obtained in Section 2.2 of \cite{AAG3} in the case of extremal Reissner-Nordstr\"om. Notice that such improved estimates are decoupled from the degeneracy at trapping and the region close to null infinity since the $(r-M)$-weights are optimized only at the horizon. 

\begin{proposition} Let $\psi$ be a solution to \eqref{eq:equation-axial-symm-extremal} and $\Phi = (\sqrt{r^2+a^2}) \psi$. Then for any $\tau_1, \tau_2$ with $\tau_1<\tau_2$ and any $\delta>0$ small enough we have that
\beaa
E^{(T)}[\psi](\tau_2) +\int_{\AA(\tau_1, \tau_2)} \Big( (r-M)^{1+\delta} |L\Phi|^2 +(r-M)^{1+\delta} |\Lb\Phi|^2+(r-M)^{3}\big(  |\nabb\Phi|^2+|\partial_t \Phi|^2 \big)\Big) \lesssim E^{(T)}[\psi](\tau_1),
\eeaa
where $\AA(\tau_1, \tau_2)=\cup_{\tau \in [\tau_1, \tau_2]} \mathcal{N}_\tau$ and $\mathcal{N}_\tau:=\Sigma_\tau \cap \{ M \leq r \leq r_0 < (1+\sqrt{2})M\}$ are null hypersurfaces that cross the event horizon but do not intersect the photon sphere.
\end{proposition}
\begin{proof} As in Proposition 2.4 in \cite{AAG2}, we first improve the weight of the $\Lb$ derivative. We define the function 
\beaa
f(r)=-\frac{1}{\eta\c [\log (r-M)^{-1}]^\eta}
\eeaa
for some small enough $\eta>0$. We integrate by parts the integral $ 2\int_{\mathcal A(\tau_1,\tau_1)}e^f (L\Lb\Phi )\c (\Lb \Phi)$ and use the equation \eqref{eq:radiation-field-final} for the radiation field $\Phi$, and obtain
\beaa
    &&\int_{\mathcal N_{\tau_2}} e^f |\Lb\Phi |^2 +\int_{\AA(\tau_1, \tau_2)}\frac{(r-M)\c e^f}{2(r^2+M^2)[\log (r-M)^{-1}]^{1+\eta}}|\Lb\Phi |^2 \\
      &=&\int_{\mathcal N_{\tau_1}} e^f |\Lb\Phi |^2+ 2\int_{\mathcal A(\tau_1,\tau_2)}e^f (L\Lb\Phi )\c (\Lb \Phi)\\
      &=&\int_{\mathcal N_{\tau_1}} e^f |\Lb\Phi |^2+ \frac 1 2\int_{\mathcal A(\tau_1,\tau_2)}e^f \Big( \frac{(r-M)^2}{(r^2+M^2)^2}\OO(\Phi )+O((r-M)^3)\Phi \Big)\c (\Lb \Phi)  
\eeaa
Integrating by parts in $\nabb_{\mathbb{S}^2}$, $\partial_t$ and $\Lb$, we obtain
\beaa
    &&\int_{\mathcal N_{\tau_2}} e^f |\Lb\Phi |^2 +\int_{\AA(\tau_1, \tau_2)}\frac{(r-M)\c e^f}{2(r^2+M^2)[\log (r-M)^{-1}]^{1+\eta}}|\Lb\Phi |^2 \\
      &=&\int_{\mathcal N_{\tau_1}} e^f |\Lb\Phi |^2+  \frac 1 4\int_{\mathcal A(\tau_1,\tau_2)} \Lb \big(e^f  \frac{(r-M)^2}{(r^2+M^2)^2}\big)\big( |\nabb_{\mathbb{S}^2}\Phi |^2+M^2\sin^2\th |\partial_t \Phi|^2\big)\\
      &&+ \frac 1 4\int   e^f  \frac{(r-M)^2}{(r^2+M^2)^2}\big( |\nabb_{\mathbb{S}^2}\Phi |^2+M^2\sin^2\th |\partial_t \Phi|^2 \big) \Big|_{r=r_0}\\
      &&+ \frac 1 2\int_{\mathcal A(\tau_1,\tau_2)}\partial_t \Big(e^f  \frac{(r-M)^2M^2\sin^2\th}{(r^2+M^2)^2}\pr_t\Phi  \c \Lb \Phi\Big) +\int_{\mathcal A(\tau_1,\tau_2)} O((r-M)^3)e^f  \Phi \c (\Lb \Phi) . 
\eeaa
The term with angular and time derivatives on $r=r_0$ can be bounded by the Morawetz bulk of Theorem 
\ref{main-theorem}. The first term on the last line can be bounded through Cauchy-Schwarz by the energy at the initial time thanks to the boundedness of the energy in \eqref{eq:energy-boundedness}. The last term can be handled by Cauchy-Schwarz using the zero-th order term of the Morawetz estimates of Theorem 
\ref{main-theorem}. 
Finally, observe that 
\beaa
-2\Lb \big(e^f  \frac{(r-M)^2}{(r^2+M^2)^2}\big)&=&e^f \pr_r( \frac{(r-M)^2}{(r^2+M^2)}) \frac{(r-M)^2}{(r^2+M^2)^2}+ e^f f' \frac{(r-M)^4}{(r^2+M^2)^3}+e^f  \frac{(r-M)^4}{(r^2+M^2)^2} \pr_r(\frac{1}{r^2+M^2}) \\
&=&e^f  \frac{2M(r+M)}{(r^2+M^2)^4} (r-M)^3- e^f \frac{1}{ [\log (r-M)^{-1}]^{\eta+1}} \frac{(r-M)^3}{(r^2+M^2)^3}-e^f  \frac{2r(r-M)^4}{(r^2+M^2)^4},
\eeaa
where the first term above dominates near the horizon. Using that in the integrated region $c \leq e^f \leq C$ for some constants $c$, $C$ and bounding
\beaa
 \frac 1 4\int_{\mathcal A(\tau_1,\tau_2)} \Lb \big(e^f  \frac{(r-M)^2}{(r^2+M^2)^2}\big)\big( |\nabb_{\mathbb{S}^2}\Phi |^2+M^2\sin^2\th |\partial_t \Phi|^2\big) \lesssim - \int_{\mathcal A(\tau_1,\tau_2)}(r-M)^3\big( |\nabb\Phi |^2+ |\partial_t \Phi|^2\big) 
\eeaa
we obtain
\beaa
    &&\int_{\mathcal N_{\tau_2}} e^f |\Lb\Phi |^2 +\int_{\AA(\tau_1, \tau_2)}(r-M)^{1+\delta}|\Lb\Phi |^2+ (r-M)^3 \big( |\nabb \Phi |^2+ |\partial_t \Phi|^2\big) \\
      &\lesssim&\int_{\mathcal N_{\tau_1}} e^f |\Lb\Phi |^2+E^{(T)}[\psi](\tau_1)
\eeaa
where for any $\delta>0$ we use that $(r-M)^\delta \lesssim \frac{1}{[\log (r-M)^{-1}]^{1+\eta}}$. 

Similarly, as in Proposition 2.4 in \cite{AAG2}, integrating by parts the integral $ -\int_{\mathcal A(\tau_1,\tau_1)} \Lb \Big((r-M)^\delta |L \Phi|^2\Big) $ and using the equation we obtain the improved weight for the $L$ derivative. Combining with the above, this concludes the proof.
\end{proof}

\end{document}